\documentclass[12pt]{amsart}
\usepackage{tikz-cd}
\usepackage{graphicx}
\usepackage{amssymb}
\usepackage{tikz}
\usetikzlibrary{matrix,arrows.meta,bending}
\usetikzlibrary{decorations.pathreplacing}
\usepackage{color}
\usepackage[numbers]{natbib}
\usetikzlibrary{arrows.meta,positioning,fit,calc}
\usepackage[a4paper, left=3cm, right=3cm, top=3cm, bottom=3cm, footskip=15mm]{geometry}

\newtheorem{theorem}{Theorem}[section]

\newtheorem{lemma}[theorem]{Lemma}

\newtheorem{corollary}[theorem]{Corollary}

\theoremstyle{definition}

\theoremstyle{remark}

\numberwithin{equation}{section}

\makeindex

\usepackage{fancyhdr}
\usepackage{calc} % ensures dimension arithmetic works robustly

\AtBeginDocument{%
  \setlength{\textwidth}{\dimexpr\paperwidth - 6cm\relax}%
  \setlength{\oddsidemargin}{\dimexpr 3cm - 1in\relax}%
  \setlength{\evensidemargin}{\dimexpr 3cm - 1in\relax}%
  \setlength{\marginparwidth}{2.5cm}%
}

\begin{document}

\title{On the flint hills series}

%    Information for first author
%\author{}
%\address{}
%\email{b.gensel@fh-kaernten.at}
%    Information for second author
\author{T. Agama}
\address{Department of Mathematics, African Institute for mathematical sciences, Ghana.}
\email{Theophilus@aims.edu.gh/emperordagama@yahoo.com}

%    General info
\subjclass[2010]{Primary 40A05; Secondary 11J82}

\date{\today}

%\dedicatory{This paper is dedicated to our advisors.}

\keywords{series; flint hill; convergence; divergence}

\begin{abstract}
In this note, we study the flint hills series of the form
\begin{align}
\sum \limits_{n=1}^{\infty}\frac{1}{(\sin^2n) n^3}\nonumber
\end{align}
via a certain method. The method essentially works by erecting certain pillars sufficiently close to the terms in the series and evaluating the series at those spots. This allows us to relate the convergence and the divergence of the series to other series that are somewhat tractable. In particular, we show that the convergence of the flint hill series relies very heavily on the condition that for any small $\epsilon>0$
\begin{align}
\bigg|\sum \limits_{i=0}^{\frac{n+1}{2}}\sum \limits_{j=0}^{i}(-1)^{i-j}\binom{n}{2i+1} \binom{i}{j}\bigg|^{2s} \leq |(\sin^2n)|n^{2s+2-\epsilon}\nonumber
\end{align}
for some $s\in \mathbb{N}$.
\end{abstract}

\maketitle

\section{Introduction}

The Flint Hills series
\begin{equation}
\sum_{n=1}^{\infty}\frac{1}{n^3\sin^2(n)}\nonumber
\end{equation}
has become a familiar open problem in elementary analytic number theory and recreational mathematics. Its appeal comes from the tension between a very simple-looking summand and a highly nontrivial global convergence question, a point emphasized early in the popular literature by Pickover \cite{pickover2002mathematics}. The problem is also notable because Alekseyev showed that convergence of the series would force a strong Diophantine consequence for $\pi$, namely an upper bound on the irrationality measure of $\pi$ \cite{alekseyev2011convergence}. This link explains why the question is so difficult: any direct attack on the series is entangled with deep information about the rational approximation.\\

The paper proposes a different route. Rather than approaching the series through irrationality measures, it develops a local trigonometric decomposition around the integer arguments of the sine function and then iterates that decomposition to compare the original sum with a family of modified series. The key auxiliary quantity is a combinatorial expression
\begin{equation}
G(n)=\sum_{i=0}^{\frac{n+1}{2}}\sum_{j=0}^{i}(-1)^{i-j}\binom{n}{2i+1}\binom{i}{j},\nonumber
\end{equation}
which arises from a multiple-angle expansion for the sine function. In the context of the paper, this quantity plays the role of a transfer factor: it allows the Flint Hills series to be rewritten, at the level of asymptotic comparison, in terms of weighted sums involving higher powers of $G(n)$.\\

The argument begins with two elementary trigonometric inputs: the standard local approximation $\sin x\sim x$ near the origin and an exact finite identity for $\sin(n\theta)$ in terms of the powers of $\sin\theta$ and $\cos\theta$. These are combined to build an iteration lemma showing that the original partial sums are asymptotically comparable to a sequence of increasingly modified partial sums. This iterative step is the technical core of the paper because it establishes a chain of asymptotic relations that can be repeated as many times as desired.\\

Using this iteration method, the paper derives its first main structural result: the Flint Hills series is equivalent, with respect to convergence or divergence, to a family of generalized series of the form
\begin{equation}
\sum_{n=1}^{\infty}\frac{(G(n))^{2s}}{n^{2s+3}\sin^2(n)},
\quad s\in\mathbb{N}.\nonumber
\end{equation}
This equivalence reformulates the original convergence problem in terms of the growth of the auxiliary factor $G(n)$ relative to the oscillatory denominator $\sin^2(n)$. The paper then presents a sufficient condition for convergence: if $G(n)$ is bounded above in a way that compensates for the singular behavior of $\sin^2(n)$, then the transformed series is dominated by a convergent $p$-series and hence the Flint Hills series converges.

\subsection{Organization of the paper} The structure of the paper is therefore straightforward. Section~2 develops the key lemmas, establishes the iteration mechanism, and proves the equivalence theorem connecting the original series to the generalized family. Section~3 formulates a convergence criterion in terms of the auxiliary combinatorial quantity $G(n)$ and concludes with the main conditional convergence statement. In this way, the paper replaces a global Diophantine question with a local-to-global analytic scheme built from trigonometric identities and asymptotic comparison.

\section{Main result}

In this section, we use a different method to study the convergence~(resp. divergence) of the flint hill series. The method works basically by erecting certain pillars, which are literally vertical lines in sufficiently small neighborhoods of the arguments of the terms in the series, and subsequently applying a certain decomposition. This allows us to obtain equivalent forms of the flint hill series at the compromise of sufficiently higher powered polynomials and certain local powered functions. Iterating the process at any given number of times, we can then obtain a general equivalent form of the flint hills series. The convergence or divergence of the flint hills series could be studied if we can say something substantial about its equivalent forms. 

\begin{lemma}\label{key lemma}
The limit holds
\begin{align}
\lim \limits_{n\longrightarrow a}\frac{\sin (n-a)}{n-a}=1\nonumber
\end{align}
equivalently 
\begin{align}
\lim \limits_{m\longrightarrow 0}\frac{\sin m}{m}=1.\nonumber
\end{align}
\end{lemma}
\bigskip

\subsection{Notation}

Throughout this paper, the limit 
$$
\lim \limits_{n\longrightarrow a}\frac{\sin (n-a)}{n-a}=1
$$
will be briefly expressed as 
$$
\frac{\sin (n-a)}{n-a}\sim 1
$$ 
in any small neighbourhood of $a$; equivalently, 
$$
\sin (n-a)\sim n-a
$$ 
in any small neighbourhood of $a$.

\begin{lemma}\label{key lemma 2}
The following identity
\begin{align}
\sin \delta&=\sin \left(\frac{\delta}{n}\right)\sum \limits_{i=0}^{\frac{n+1}{2}}\sum \limits_{j=0}^{i}(-1)^{i-j}\binom{n}{2i+1} \binom{i}{j}\cos^{n-2(i-j)-1}\left(\frac{\delta}{n}\right)\nonumber
\end{align}
holds for any $\delta>0$ and $n\in \mathbb{N}$ with $n>1$.
\end{lemma}

\begin{proof}
This identity is easily obtained by writing 
\begin{align}
\sin \delta&=\sin \left(\frac{\delta \cdot n}{n}\right)\nonumber
\end{align}
and applying the trigonometric identity 
\begin{align}
\sin (n\theta)=\sin \theta \sum \limits_{i=0}^{\frac{n+1}{2}}\sum \limits_{j=0}^{i}(-1)^{i-j}\binom{n}{2i+1}\binom{i}{j}\cos^{n-2(i-j)-1}\theta\nonumber
\end{align}
which can be accessed on the Wikipedia page and due to  Francois Viete.
\end{proof}

\begin{lemma}\label{The iteration method}
The following asymptotic holds
\begin{align}
\sum \limits_{n=1}^{k}\frac{1}{(\sin^2n) n^3} \sim \sum \limits_{n=1}^{k}\frac{(G(n))^2}{(\sin^2n) n^5}\cdots \sim \sum \limits_{n=1}^{k}\frac{(G(n))^{2s}}{(\sin^2n) n^{2s+3}}\nonumber
\end{align} 
where 
\begin{align}
G(n)=\sum \limits_{i=0}^{\frac{n+1}{2}}\sum \limits_{j=0}^{i}(-1)^{i-j}\binom{n}{2i+1} \binom{i}{j}\nonumber
\end{align}
for all $s\geq 1$ with $s\in \mathbb{N}$.
\end{lemma}

\begin{proof}
Using the lemma \ref{key lemma} and the decomposition 
$$
\sin (n-a)=(\sin n)\left(\cos a-\frac{(\sin a)\cos n}{\sin n}\right)
$$
we obtain the relation 
\begin{align}
1\sim \frac{\sin (n-a)}{n-a} \sim \frac{(\sin n)(\cos a-\frac{(\sin a)\cos n}{\sin n})}{n-a}\nonumber
\end{align}
in any small neighbourhood of $a$, so that by rearranging, we deduce
\begin{align}
\sin n \sim \frac{n-a}{(\cos a-\frac{(\sin a)\cos n}{\sin n})}\label{major}
\end{align}
in any small neighbourhood of $a$. Plugging \eqref{major} into the finite sum,  we can write
\begin{align}
\sum \limits_{n=1}^{k}\frac{1}{(\sin^2n) n^3}&\sim \sum \limits_{\substack{n=1\\a=n+\delta\\\delta\longrightarrow 0^{+}}}^{k}\frac{(\cos a-\frac{(\sin a)\cos n}{\sin n})^2}{n^2(1-\frac{a}{n})^2n^3}\nonumber \\&=\sum \limits_{\substack{n=1\\\delta\longrightarrow 0^{+}}}^{k}\frac{(\cos (n+\delta)-\frac{(\sin (n+\delta))\cos n}{\sin n})^2}{n^2(\frac{n+\delta}{n}-1)^2n^3}\nonumber \\&=\sum \limits_{\substack{n=1\\\delta \longrightarrow 0^{+}}}^{k}\frac{(G(n))^2\sin^2(\frac{\delta}{n})\cos^{2n-4(i-j)-2}(\frac{\delta}{n})}{(\frac{\delta}{n})^2(\sin^2n)n^5}\nonumber \\&\sim \sum \limits_{n=1}^{k}\frac{(G(n))^2}{(\sin^2n) n^5}\nonumber
\end{align}
using Lemma \ref{key lemma} and \ref{key lemma 2}, where 
\begin{align}
G(n)=\sum \limits_{i=0}^{\frac{n+1}{2}}\sum \limits_{j=0}^{i}(-1)^{i-j}\binom{n}{2i+1} \binom{i}{j}.\nonumber
\end{align}
Repeating the argument on $\sin n$ in the deduced finite sum, we obtain 
\begin{align}
\sum \limits_{n=1}^{k}\frac{(G(n))^2}{(\sin^2n) n^5}\sim \sum \limits_{n=1}^{k}\frac{(G(n))^4}{(\sin^2n) n^7}.\nonumber
\end{align}
Iterating the argument in this manner, we deduce the claimed chain of asymptotic.
\end{proof}
\bigskip

\begin{tikzpicture}[
    >=Latex,
    font=\small,
    every node/.style={align=center},
    pillar/.style={very thick,black!75},
    neigh/.style={draw=blue!55,fill=blue!8,rounded corners=2pt},
    box/.style={draw=black!55,rounded corners=2pt,fill=white,inner sep=4pt}
]

%========================
% Panel A: geometric intuition
%========================
\begin{scope}[xshift=0cm,yshift=0cm]
  \node[font=\bfseries] at (3.4,4.35) {\Large A. Local pillar picture};

  % axes
  \draw[->] (-0.2,0) -- (7.0,0) node[right] {$x$};
  \draw[->] (0,-0.2) -- (0,3.7) node[above] {$f(x)$};

  % neighborhood around n
  \fill[blue!8] (2.35,0) rectangle (4.15,3.25);
  \draw[blue!55,dashed] (2.35,0) -- (2.35,3.25);
  \draw[blue!55,dashed] (4.15,0) -- (4.15,3.25);
  \node[blue!60!black] at (3.25,3.43) {$U_\delta(n)$};

  % stylized local oscillation / singular behaviour near the argument
  \draw[thick,domain=0.45:6.55,samples=300,smooth]
    plot (\x,{0.25 + 0.32/(0.12 + abs(sin(1.7*\x r)))/(1+0.18*\x)});
  \node[anchor=west] at (4.95,2.55) {$f(x)=\dfrac{1}{\sin^2 x\,x^3}$};

  % the pillar at n
  \draw[pillar] (3.25,0) -- (3.25,2.85);
  \fill (3.25,2.85) circle (1.2pt);
  \node[above] at (3.25,2.85) {$\bigl(n,f(n)\bigr)$};
  \node[below] at (3.25,0) {$n$};
  \draw[decorate,decoration={brace,amplitude=4pt,mirror}] (3.25,0.08) -- (4.00,0.08);
  \node[below=6pt] at (3.62,0) {$\delta$};

  % nearby evaluation point a=n+delta
  \fill (3.92,1.55) circle (1.2pt);
  \draw[densely dashed,->,thick] (3.92,1.55) -- (3.92,0);
  \node[below] at (3.92,0) {$a=n+\delta$};
  \node[above right] at (4.02,1.55) {$\text{evaluate in a small neighborhood}$};

  % local linearization label
  \node[box,anchor=west] at (0.25,0.55)
    {$\displaystyle \frac{\sin(n-a)}{n-a}\sim 1\quad(a\to n)$};
  \node[box,anchor=west] at (0.25,0.03)
    {$\displaystyle \sin(n-a)\sim n-a$};

  % arrow to the right within the panel
  \draw[->,thick] (5.8,0.9) -- (6.55,0.9);
  \node[anchor=west,align=left,text width=4.6cm] at (5.95,1.15) {replace the local term by its\\[-1pt] nearby asymptotic model};
\end{scope}

%========================
% Panel B: algebraic iteration
%========================
\begin{scope}[xshift=8.2cm,yshift=0cm]
  \node[font=\bfseries] at (4.0,4.35) {\Large B. Iterative transfer of powers};

  % boxes
  \node[box,minimum width=6.8cm,minimum height=0.95cm] (start) at (4.0,3.35)
    {$\displaystyle \sum_{n\ge 1}\frac{1}{\sin^2 n\,n^3}$};
  \node[box,minimum width=6.8cm,minimum height=0.95cm,below=0.9cm of start] (mid)
    {$\displaystyle \sum_{n\ge 1}\frac{G(n)^2}{\sin^2 n\,n^5}$};
  \node[box,minimum width=6.8cm,minimum height=0.95cm,below=0.9cm of mid] (end)
    {$\displaystyle \sum_{n\ge 1}\frac{G(n)^{2s}}{\sin^2 n\,n^{2s+3}}$};

  % arrows
  \draw[->,very thick] (start) -- node[right,align=left] {one local iteration} (mid);
  \draw[->,very thick] (mid) -- node[right,align=left] {repeat} (end);

  % explanation box
  \node[box,align=left,anchor=north west,text width=6.9cm] at (0.0,1.35) {
    \textbf{Interpretation.}\\[-1pt]
    Each pillar is a local replacement of the original term by an equivalent nearby expression.\\[-1pt]
    After one pass, the power on the auxiliary factor increases; after $s$ passes, the general form appears.
  };

  % formula callout
  \node[box,align=left,anchor=south west,text width=6.9cm] at (0.0,0.15) {
    \(\displaystyle G(n)=\sum_{i=0}^{\frac{n+1}{2}}\sum_{j=0}^{i}(-1)^{i-j}\binom{n}{2i+1}\binom{i}{j}.\)
  };
\end{scope}

\end{tikzpicture}

\begin{corollary}\label{asymptotic}
We have
\begin{align}
G(n)=\sum \limits_{i=0}^{\frac{n+1}{2}}\sum \limits_{j=0}^{i}(-1)^{i-j}\binom{n}{2i+1} \binom{i}{j}\sim n.\nonumber
\end{align}
\end{corollary}

\begin{proof}
The asymptotic follows from Theorem \ref{The iteration method}.
\end{proof}
\bigskip

Now, we show that we can study the convergence of the flint hill series by examining its equivalent forms in the following result.

\begin{theorem}\label{equivalence theorem}
The flint hills series 
\begin{align}
\sum \limits_{n=1}^{\infty}\frac{1}{(\sin^2n) n^3}\nonumber
\end{align}
is convergent (resp. divergent) if and only if 
\begin{align}
\sum \limits_{n=1}^{\infty}\frac{(G(n))^{2s}}{(\sin^2n) n^{2s+3}}\nonumber
\end{align}
is convergent (resp. divergent) for some $s\geq 1$ with $s\in \mathbb{N}$, where 
\begin{align}
G(n)=\sum \limits_{i=0}^{\frac{n+1}{2}}\sum \limits_{j=0}^{i}(-1)^{i-j}\binom{n}{2i+1} \binom{i}{j}.\nonumber
\end{align}
\end{theorem}

\begin{proof}
Using Lemma \ref{The iteration method}, we can write the asymptotic
\begin{align}
\sum \limits_{n=1}^{k}\frac{1}{(\sin^2n) n^3} \sim \sum \limits_{n=1}^{k}\frac{(G(n))^2}{(\sin^2n) n^5}\cdots \sim \sum \limits_{n=1}^{k}\frac{(G(n))^{2s}}{(\sin^2n) n^{2s+3}}\nonumber
\end{align}
for all $s\geq 1$ with $s\in \mathbb{N}$ where 
\begin{align}
G(n)=\sum \limits_{i=0}^{\frac{n+1}{2}}\sum \limits_{j=0}^{i}(-1)^{i-j}\binom{n}{2i+1} \binom{i}{j}\nonumber
\end{align}
so that 
\begin{align}
\sum \limits_{n=1}^{\infty}\frac{1}{(\sin^2n) n^3}\nonumber
\end{align}
is convergent (resp. divergent) if and only 
\begin{align}
\sum \limits_{n=1}^{\infty}\frac{(G(n))^{2s}}{(\sin^2n) n^{2s+3}}\nonumber
\end{align}
is convergent (resp. divergent) for some $s\geq 1$ with $s\in \mathbb{N}$.
\end{proof}
\bigskip

\section{The convergence or divergence criterion}

Here, we introduce a criterion for determining the convergence of the flint hill series. The following could be considered as a test tool for deciding on the convergence or divergence of the flint hill series, avoiding studies of irrationality measure of $\pi$ which is generally a harder problem.

\begin{theorem}\label{conditional theorem}
If for any small $\epsilon>0$
\begin{align}
\bigg|\sum \limits_{i=0}^{\frac{n+1}{2}}\sum \limits_{j=0}^{i}(-1)^{i-j}\binom{n}{2i+1} \binom{i}{j}\bigg|^{2s} \leq |(\sin^2n)|n^{2s+2-\epsilon}\nonumber
\end{align}
for some $s\in \mathbb{N}$, then the flint hill series
\begin{align}
\sum \limits_{n=1}^{\infty}\frac{1}{(\sin^2n) n^3}\nonumber
\end{align}
converges.
\end{theorem}

\begin{proof}
For any small $\epsilon>0$, we can write 
\begin{align}
\sum \limits_{n=1}^{\infty}\frac{(G(n))^{2s}}{(\sin^2n) n^{2s+3}}&\leq \sum \limits_{n=1}^{\infty}\frac{1}{n^{1+\epsilon}}<\infty \nonumber
\end{align}
under the condition 
\begin{align}
\bigg|\sum \limits_{i=0}^{\frac{n+1}{2}}\sum \limits_{j=0}^{i}(-1)^{i-j}\binom{n}{2i+1} \binom{i}{j}\bigg|^{2s} \leq |(\sin^2n)|n^{2s+2-\epsilon}.\nonumber
\end{align}
By Theorem \ref{equivalence theorem}, the flint hill series
\begin{align}
\sum \limits_{n=1}^{\infty}\frac{1}{(\sin^2n) n^3}<\infty \nonumber
\end{align}
and so it converges.
\end{proof}
\bigskip

%%%%%%%%%%%%%%%%%%%%%%%%%%%%%%%%%%%%%%%%%%%%%%%%%%%%%%%%%%%%%%%%%%%%%%%%%%%%%%%%%%%%%%%%%%%%%%%%%%%
\rule{100pt}{1pt}

\bibliographystyle{amsplain}

\end{document}